\documentclass[12pt,reqno]{amsart}
\usepackage[utf8]{inputenc}
\usepackage[T1]{fontenc}
\usepackage[usenames, dvipsnames]{color}
\usepackage{ulem}

\usepackage{dsfont, amsfonts, amsmath, amssymb,amscd, stmaryrd, latexsym, amsthm, dsfont}
\usepackage[frenchb,english]{babel}
\usepackage{enumerate}
\usepackage{longtable}
\usepackage{geometry}
\usepackage{float}
\usepackage{tikz}
\usetikzlibrary{shapes,arrows}
\usepackage{float}
\usepackage{tikz}
\usepackage{xypic}
\usetikzlibrary{shapes,arrows}

\newtheorem{theorem}{Theorem}[section]
\newtheorem{lemma}[theorem]{Lemma}
\newtheorem{proposition}[theorem]{Proposition}
\newtheorem{corollary}[theorem]{Corollary}

\theoremstyle{remark}
\newtheorem{remark}[theorem]{\bf Remark}

\usepackage{hyperref}
\hypersetup{
	colorlinks=true,
	urlcolor=blue,
	citecolor=blue}
\def\NN{\mathds{N}}
\def\RR{\mathbb{R}}
\def\QQ{\mathbb{Q}}
\def\CC{\mathbb{C}}
\def\ZZ{\mathbb{Z}}
\def\kk{\mathds{k}}
\begin{document}
	
\def\NN{\mathbb{N}}
\def\RR{\mathds{R}}
\def\HH{I\!\! H}
\def\QQ{\mathbb{Q}}
\def\CC{\mathds{C}}
\def\ZZ{\mathbb{Z}}
\def\DD{\mathds{D}}
\def\OO{\mathcal{O}}
\def\kk{\mathds{k}}
\def\KK{\mathbb{K}}
\def\ho{\mathcal{H}_0^{\frac{h(d)}{2}}}
\def\LL{\mathbb{L}}
\def\L{\mathds{k}_2^{(2)}}
\def\M{\mathds{k}_2^{(1)}}
\def\k{\mathds{k}^{(*)}}
\def\l{\mathds{L}}

\selectlanguage{english}
\title[Units and  $2$-class field towers...]{Units and  $2$-class field towers  of some multiquadratic number fields}

\author[M. M. Chems-Eddin]{Mohamed Mahmoud Chems-Eddin}
\address{Mohamed Mahmoud CHEMS-EDDIN: Mohammed First University, Mathematics Department, Sciences Faculty, Oujda, Morocco }
\email{2m.chemseddin@gmail.com}

\author[A. Zekhnini]{Abdelkader Zekhnini}
\address{Abdelkader Zekhnini: Mohammed First University, Mathematics Department, Pluridisciplinary faculty, Nador, Morocco}
\email{zekha1@yahoo.fr}

\author[A. Azizi]{Abdelmalek Azizi}
\address{Abdelmalek Azizi: Mohammed First University, Mathematics Department, Sciences Faculty, Oujda, Morocco }
\email{abdelmalekazizi@yahoo.fr}

\subjclass[2000]{11R04, 11R27, 11R29, 11R37.}
\keywords{Multiquadratic number fields,  unit groups, $2$-class groups, Hilbert $2$-class field towers.}

 \begin{abstract} 
 	In this paper, we investigate  the unit  groups,  the $2$-class groups,  the  $2$-class field towers and the structures of  the   second $2$-class groups   of some  multiquadratic number fields of degree $8$ and $16$.

 	\keywords{Multiquadratic number fields,  unit groups, $2$-class groups, Hilbert $2$-class field towers.}
 \end{abstract}
 \maketitle
 
 \section{Introduction}
 \label{Sec:1}
 
 Let $k$ be  an algebraic number field and $\mathrm{Cl}_2(k)$  its $2$-class group, that
 is the  $2$-Sylow subgroup of the ideal class group $\mathrm{Cl}(k)$  of $k$. Let  $k^{(1)}$  be the Hilbert $2$-class field of $k$, that is the maximal unramified (including the infinite primes) abelian field extension of $k$ whose degree over $k$ is a $2$-power. Put $k^{(0)} = k$ and let $k^{(i)}$ denote the Hilbert $2$-class field of $k^{(i-1)}$
 for any integer $i\geq 1$. Then the sequence of fields
 $$k=k^{(0)} \subset k^{(1)} \subset  k^{(2)}  \subset \cdots\subset k^{(i)} \cdots $$
 is  called  the $2$-class field tower of $k$. If for all $i\geq1$,  $k^{(i)}\neq k^{(i-1)}$, the tower is said to be infinite, otherwise the tower is said to be  finite, and the minimal integer $i$ satisfying the condition $k^{(i)}= k^{(i-1)}$ is called the length of the tower. Unfortunately, deciding  whether or not the $2$-class field tower of a number field $k$ is finite is still an open problem and  there is no known method to study this   finiteness. However, it is known that if the rank of $\mathrm{Cl}_2(k^{(1)})\leq 2$, then by group theoretical meaning
 the tower is finite and its length is $\leq 3$ (cf.\cite{Bl}).  Furthermore, this  problem is closely related to  the structure of the Galois group of the tower. In particular, for $\mathrm{Cl}_2(k)$ being cyclic, the Hilbert $2$-class field tower of $k$ terminates at the first step $k^{(1)}$, whereas for $\mathrm{Cl}_2(k)$ being of type  $(2,2)$, the Hilbert $2$-class field tower of $k$ terminates in at most two steps and the structure of the Galois group $G_k=\mathrm{Gal}(k^{(2)}/k)$ is closely related to the capitulation problem in the unramified quadratic extensions of $k$ (cf. \cite{kisilvsky}), so to the unit groups of these extensions. In fact the number of classes which capitulate in these quadratic extensions of $k$ is given in terms of   their unit groups (cf. \cite{Heider}). In the literature, the most  studies done in this vein concern quadratic or biquadratic fields (e.g. \cite{kisilvsky,benja,Be05}). In this paper, we are interested in investigating  the $2$-class field towers of some multiquadratic number fields related to
 the imaginary triquadratic number field $\mathbb{Q}(\zeta_8,\sqrt{d})$ whenever its $2$-class group is of type $(2, 2)$, where $\zeta_8$
 is  a primitive $8$-th root of unity and $d$ is an odd positive square free integer.

 The layout of this paper is the following. In \S \ref{section 0}, we quote some properties of $2$-groups $G$ satisfying $G/G'\simeq (2, 2)$. Next, in \S \ref{section 01}, we characterize the $2$-class groups of some imaginary multiquadratic number fields and we compute their $2$-class numbers.
 In \S \ref{section 1}, involving  some technical computations, we give  unit groups of some multiquadratic number fields of degree $8$ and $16$.
 Thereafter, and as applications, in \S \ref{section 2},   we shall investigate the Hilbert $2$-class field tower   of
 some families of imaginary triquadratic  number fields; and then we deduce the capitulation behaviors  in the unramified quadratic extensions of these fields.
 \section*{Notations}
 Let $k$ be a number field. Throughout this paper we shall respect the following notations:
 \begin{enumerate}[$\bullet$]
 	\item $p$, $p'$ and $q$: Three    prime integers,
 	\item  $\mathrm{Cl}_2(k)$: The $2$-class group of $k$,
 	\item $h_2(k)$: The $2$-class number of $k$,
 	\item $h_2(d)$: The $2$-class number of the quadratic field $\mathbb{Q}(\sqrt{d})$,
 	\item $\varepsilon_d$: The fundamental unit of the quadratic field $\mathbb{Q}(\sqrt{d})$,
 	\item $E_k$: The unit group of $k$,
 	\item FSU: Abbreviation of ``fundamental system of units'',
 	\item $k^{(1)}$: The Hilbert $2$-class field of $k$,
 	\item $k^{(2)}$: The Hilbert $2$-class field of $k^{(1)}$,
 	\item $G_k$: The Galois group of $k^{(2)}/k$, i.e. $Gal(k^{(2)}/k)$,
 	\item $k^{*}$: The genus field of $k$,
 	\item $k^{+}$: The maximal real  subfield of $k$, whenever $k$ is imaginary,
 	\item $q(k)=[E_{k}: \prod_{i}E_{k_i}]$: The unit index of $k$, if $k$ is multiquadratic and  $k_i$ are  the  quadratic subfields of $k$,
 	\item $N_{k'/k}$: The norm map of an extension $k'/k$.
 \end{enumerate}

 \section{\textbf{Some preliminary results of   group theory}}\label{section 0}	
 Let $Q_m$,   $D_m$,  and $S_m$ denote  the quaternion,   dihedral and semidihedral groups respectively,  of order $2^m$,  where $m\geq3$ and  $m\geq4$ for $S_m$. In addition,  let $A$ denote the Klein four-group. Each of these groups is generated by two elements $x$ and $y$,  and admits  a representation by generators and relations as follows:
 $$\begin{array}{ll}
 A= \{x,  y\ |\ x^2 = y^2 = 1, \ y^{-1}xy = x\}, \\
 Q_m = \{x,  y \ |\ x^{2^{m-2}} = y^2 = a,  a^2 = 1, \ y^{-1}xy = x^{-1}\}, \\
 D_m = \{x,  y\ | \ x^{2^{m-1}} = y^2 = 1, \  y^{-1}xy = x^{-1}\}, \\
 S_m = \{x,  y\ |\ x^{2^{m-1}} = y^2 = 1, \  y^{-1}xy = x^{2^{m-2}-1}\}.
 \end{array}$$
 We shall recall some well know properties of $2$-groups $G$ such that $G/G'$ is of type $(2,  2)$,  where $G'$ denotes the  commutator subgroup of $G$. For more details about these properties, we refer the reader to \cite[pp. 272-273]{kisilvsky},  \cite[p. 162]{benja}  and \cite[Chap. 5]{Go}.

 Let $k$ be  an algebraic number field and   $\mathrm{Cl}_2(k)$ the $2$-Sylow subgroup of its ideal class group $\mathrm{Cl}(k)$. Let  $k^{(1)}$ (resp. $k^{(2)}$) be the first (resp. second) Hilbert $2$-class field of $k$ and  put $G=\mathrm{Gal}(k^{(2)}/k)$,  then if $G'$ denotes the commutator subgroup of $G$,  we have by class field theory $G'\simeq\mathrm{Gal}(k^{(2)}/k^{(1)})$ and  $G/G'\simeq\mathrm{Gal}(k^{(1)}/k)\simeq\mathrm{Cl}_2(k)$. Assume in all what follows that $\mathrm{Cl}_2(k)$ is of type $(2,  2)$,  then it is known  that $G$ is isomorphic to $A$,  $Q_m$,  $D_m$   or $S_m$.

 Let $x$ and $y$ be as above. Note that
 the  commutator subgroup $G'$ of $G$  is always cyclic and $G'= \langle x^2\rangle$. The group $G$ possesses exactly three subgroups of index $2$ which are:
 $$H_1 = \langle x\rangle, \; H_2 = \langle x^2,  y\rangle, \; H_3 = \langle x^2,  xy\rangle.$$
 Note that for the two cases $Q_3$ and $A$, each $H_i$ is cyclic. For the case $D_m$, with $m>3$,  $H_2$ and $H_3$ are also dihedral. For $Q_m$, with $m>3$,  $H_2$ and $H_3$ are quaternion. Finally for $S_m$,
 $H_2$ is dihedral whereas $H_3$ is quaternion. Furthermore,  if $G$ is isomorphic to $A$ (resp.  $Q_3$),  then the subgroups $H_i$ are cyclic of order $2$ (resp. $4$).
 If $G$ is isomorphic to $ Q_m$, with $m>3$,  $ D_m$ or $ S_m$,  then   $H_1$ is cyclic and $H_i/H_i'$ is of type $(2, 2)$ for $i\in \{2,  3\}$,  where $H_i'$ is the commutator subgroup of $H_i$.
 
 Let  $F_i$ be the subfield of $k^{(2)}$ fixed by $H_i$,  where $i\in\{1,  2,  3\}$.
 If $k^{(2)}\not=k^{(1)}$,  $\langle x^4\rangle$ is the unique subgroup of $G'$ of index $2$.
 Let  $L$ ($L$ is defined only if $k^{(2)}\not=k^{(1)}$)  be
 the subfield of $k^{(2)}$ fixed by $\langle x^4\rangle$. Then $F_1$,  $F_2$ and $F_3$ are the three quadratic subextensions of $k^{(1)}/k$ and $L$ is the unique subfield of $k^{(2)}$ such that $L/k$ is a nonabelian Galois extension of degree $8$.
 
 We continue by recalling the definition of Taussky's conditions $A$ and $B$ (\cite{Tau}).
 Let $k'$ be a cyclic unramified extension of a number field  $k$ and $j$ denotes the basic homomorphism:
 $j_{k'/k}:\mathrm{Cl}({k})\longrightarrow \mathrm{Cl}({k'}), $
 induced by extension of ideals from $k$ to $k'$. Thus,  we say
 \begin{enumerate}[\rm 1.]
 	\item $k'/k$ satisfies condition $A$ if and only if  $|ker(j_{k'/k})\cap N_{k'/k}(\mathrm{Cl}(k'))|>1$.
 	\item $k'/k$ satisfies condition $B$ if and only if  $|ker(j_{k'/k})\cap N_{k'/k}(\mathrm{Cl}(k'))|=1$.
 \end{enumerate}
 \noindent Set $j_{F_i/k}=j_i$,  $i=1, 2, 3$. Then we have:
 \begin{theorem}[\text{\cite[Theorem 2]{kisilvsky}}]~\ \label{1}
 	\begin{enumerate}[\rm1.]
 		\item If $k^{(1)}=k^{(2)}$,  then $F_i$ satisfy  condition $A$,  $|ker(j_i)|=4$,  for $i=1, 2, 3$,  and $G$ is abelian of type $(2,  2)$.
 		\item If $Gal(L/k)\simeq Q_3$,  then $F_i$ satisfy condition $A$ and $|ker(j_i)|=2$ for $i=1, 2, 3$ and $G\simeq Q_3$.
 		\item  If $Gal(L/k)\simeq D_3$,  then $F_2$,  $F_3$ satisfy condition $B$ and $|ker j_2|=|ker j_3|=2$. Furthermore,  if $F_1$ satisfies condition $B$,  then $|ker j_1|=2$ and $G\simeq S_m$,  if $F_1$ satisfies condition $A$ and $|ker j_1|=2$ then $G\simeq Q_m$. If $F_1$ satisfies condition $A$ and $|ker j_1|=4$ then $G\simeq D_m$.
 	\end{enumerate}
 	These results are summarized in the following table (see Table \ref{tablecapi}).
 	\begin{table}[H]
 		
 		$${ \begin{tabular}{c|c|c|c}
 			\hline
 			\hline
 			$|ker\; j_1|\;(A/B)$& $|ker\; j_2|\;(A/B)$ & $|ker\; j_3|\;(A/B)$& \;\;\;G\;\;\;\\
 			\hline
 			$4$ & $4$ & $4$ & $(2, 2)$\\
 			\hline
 			$2A$ & $2A$ & $2A$ & $Q_3$\\
 			\hline
 			$4$  & $2B$ & $2B$ & $D_m,  m\geq3$\\
 			\hline
 			$2A$ & $2B$ & $2B$ &  $Q_m,  m>3$\\
 			\hline
 			$2B$ & $2B$ & $2B$ &  $S_m,  m>3$\\
 			\hline
 			\hline
 			\end{tabular}}$$
 		\caption{Capitulation types}\label{tablecapi}
 	\end{table}
 \end{theorem}
 
 By Theorem \ref{1} and  group theoretic properties quoted in the beginning of this section,  one can easily deduce the following remark.
 \begin{remark}\label{rmk 1 preliminaries}
 	The $2$-class groups of the three unramified quadratic  extensions of $k$ are cyclic if and only if  $k^{(1)}=k^{(2)}$ or  $k^{(1)}\not=k^{(2)}$ and $G\simeq Q_3$. In the other cases the $2$-class group of only one unramified quadratic  extension is cyclic and the others are of type $(2,  2)$.
 \end{remark}
 
 \begin{proposition}\label{prop the order of G_k}
 	Let $K$ be a number field and let $L$ be an abelian  unramified  $2$-extension of $K$. If $G_L=\mathrm{Gal}(L^{(2)}/L)$   is abelian,  then $L$ and $K^{(1)}$ have the same Hilbert $2$-class field
 	$($i.e. $K^{(2)}=L^{(1)})$. Furthermore,   $|G_K|=[L:K]\cdot h_2(L)$. In particular,  if the $2$-class group of $K^*$ $($i.e., the genus field of $K)$ is cyclic,  then
 	$|G_K|=[K^*:K]\cdot h_2(K^*)$.
 \end{proposition}
 \begin{proof}The fact that $L/K$ is unramified   implies that
 	$K\subset L\subset {K}^{(1)}\subset {L}^{(1)}\subset {K}^{(2)}\subset {L}^{(2)}.$
 	Since $G_L$ is abelian,  we deduce that ${L}^{(1)}={L}^{(2)}$. So $K^{(2)}=L^{(1)}$, hence the first equality. As the $2$-class group of  $K^*$  is cyclic, so the $2$-class field tower of $K$ terminates at the first step, which completes the proof.
 \end{proof}
 
 \section{\textbf{$2$-class groups of some multiquadratic number fields}}\label{section 01}
 Let $d$ be an odd positive square-free integer and $\zeta_8$  a primitive $8$-th root of unity. Note that the triquadratic field $L_d=\mathbb{Q}(\zeta_8,\sqrt{d})$  is the first step of
 both the cyclotomic $\mathbb{Z}_2$-extension and a non-cyclotomic $\mathbb{Z}_2$-extension (e.g. the non-cyclotomic  $\mathbb{Z}_2$-extension of
 $\mathbb{Q}(\sqrt{-1})$ (cf. \cite{Hubbard}) shifted by $\sqrt{d}$) of the biquadratic field
 $\mathbb{Q}(\sqrt{-1},\sqrt{d})$. This fact makes it of a particular importance from the triquadratic number fields.
 In this subsection  and the next section, we will make some preparations on the $2$-class groups and unit groups of some multiquadratic number fields that will help to study the $2$-class field towers    and the second $2$-class groups of all   fields  $L_d=\mathbb{Q}(\zeta_8,\sqrt{d})$, such that
 $2$-class group is of type $(2, 2)$,  together with some different families of mutiquadratic number fields in the last section of this paper.
 For this, we recall  the following results of our earlier paper \cite[Theorem 5.7]{chemsZkhnin1}.
 Let $p$, $p'$ and $q$ be three  prime integers. The $2$-class group of $L_d=\QQ(\sqrt d, \zeta_8)$ is of type $(2, 2)$ if and only if $d$ takes one of the following forms:
 \begin{eqnarray}\label{cond 1}
 d=pq \text{ with } p\equiv -q\equiv 1\pmod 4, 	\left(\dfrac{2}{p}\right)=-1,  \left(\dfrac{2}{q}\right)=1 \text{ and }  \left(\dfrac{p}{q}\right)=-1.
 \end{eqnarray}
 \begin{eqnarray}\label{cond 2}
 d=pq \text{ with } p\equiv q\equiv -1\pmod 4, 	\left(\dfrac{2}{p}\right)=-1,  \left(\dfrac{2}{q}\right)=1 \text{ and }  \left(\dfrac{p}{q}\right)=-1 ,
 \end{eqnarray}
 \begin{eqnarray}\label{cond 3}
 d=p'\text{ with } p'\equiv 1 \pmod{16}  \text{ and }  \left(\frac{2}{p'}\right)_4\not=\left(\frac{p'}{2}\right)_4 .
 \end{eqnarray}
 
 \noindent Consider the following notations
 \begin{enumerate}[\rm1.]
 	\item  $L_{pq}=\mathbb{Q}(\sqrt2, \sqrt{qp}, \sqrt{-1}) $.
 	\item $L_{pq}^*=\QQ(\sqrt 2,  \sqrt{p},  \sqrt{q},  \sqrt{-1})$  the   genus field of $L_{pq}$.
 	\item    $L_{p'}=\mathbb{Q}(\sqrt2, \sqrt{p'}, \sqrt{-1}) $.
 	\item $F_{pq}=\QQ(\sqrt{2p},  \sqrt{2q}, \sqrt{-2})$ or  $\QQ(\sqrt{p},  \sqrt{2q}, \sqrt{-2})$,  according to whether   $p$ and $q$ verify conditions {\eqref{cond 1}} or   {\eqref{cond 2}} .
 	\item $K_{pq}=\QQ(\sqrt{p},  \sqrt{q}, \sqrt{-2})$ or  $\QQ(\sqrt{q},  \sqrt{2p}, \sqrt{-2})$,  according to whether   $p$ and $q$ verify conditions {\eqref{cond 1}} or  {\eqref{cond 2}}.
 	\item $k_{pq}=\mathbb{Q}(\sqrt{-2}, \sqrt{pq})$ or $\mathbb{Q}(\sqrt{-2}, \sqrt{2pq})$ according to whether $p$ and $q$ verify conditions {\eqref{cond 1}} or  {\eqref{cond 2}}.
 	\item $m\geq 2$ the positive integer satisfying  $h_2(-2q)=2^m$ (cf. \cite{connor88}).
 \end{enumerate}
 
 \noindent Let us start with some lemmas that we shall use in what follows.
 \begin{lemma}[\cite{azizunints99}]\label{Lemme azizi} Let $K_0$ be a real number field, $K=K_0(i)$ a quadratic extension of $K_0$, $n\geq 2$ be an integer and $\xi_n$ a $2^n$-th primitive root of unity, then
 	$	\xi_n=\frac{1}{2}(\mu_n+\lambda_ni)$, where $\mu_n=\sqrt{2+\mu_{n-1}}$, $\lambda_n=\sqrt{2-\mu_{n-1}}$, $\mu_2=0$, $\lambda_2=2$ and $\mu_3=\lambda_3=\sqrt{2}$. Let $n_0$ be the greatest
 	integer such that $\xi_{n_0}$ is contained in $K$, $\{\varepsilon_1, ..., \varepsilon_r\}$ a fundamental system of units of $K_0$ and $\varepsilon$ a unit of $K_0$ such that
 	$(2+\mu_{n_0})\varepsilon$ is a square in $K_0$(if it exists). Then a fundamental system of units of $K$ is one of the following systems:
 	\begin{enumerate}[\rm 1.]
 		\item $\{\varepsilon_1, ..., \varepsilon_{r-1}, \sqrt{\xi_{n_0}\varepsilon } \}$ if $\varepsilon$ exists,  in this case $\varepsilon=\varepsilon_1^{j_1}...\varepsilon_{r-1}^{j_1}\varepsilon_r$,
 		where $j_i\in \{0, 1\}$.
 		\item $\{\varepsilon_1, ..., \varepsilon_r \}$ elsewhere.
 		
 	\end{enumerate}
 \end{lemma}

 \begin{lemma}\label{lemma units of biquadratic fields}
 	Let $p$ and $q$ be  two primes satisfying   conditions   {\eqref{cond 1}}. Then,
 	\begin{enumerate}[\rm 1.]	
 		\item A FSU of $\QQ(\sqrt{p}, \sqrt{q})$ is $\{\varepsilon_{p},  \varepsilon_{q},  \sqrt{\varepsilon_{q}\varepsilon_{pq}}\}.$
 		\item  A FSU of  $\QQ(\sqrt{p}, \sqrt{2q})$ is   $\{\varepsilon_{p},  \varepsilon_{2q},  \sqrt{\varepsilon_{2pq}}\}$.
 	\end{enumerate}
 \end{lemma}
 \begin{proof} To prove this lemma,  we will use the algorithm described in  \cite{wada}.
 	Let $\varepsilon_{pq}=x+y\sqrt{pq}$ for some integers $x$ and $y$. Since $N(\varepsilon_{pq})=1$,  then $x^2-1=y^2pq$. Hence by     the unique  factorization in $\ZZ$  there exist $y_1$,  $y_2$ in $\ZZ$ such that
 	$$(1):\ \left\{ \begin{array}{ll}
 	x\pm1=y_1^2\\
 	x\mp1=pqy_2^2,
 	\end{array}\right. \quad
 	(2):\ \left\{ \begin{array}{ll}
 	x\pm1=py_1^2\\
 	x\mp1=qy_2^2,
 	\end{array}\right.\quad
 	(3):\ \left\{ \begin{array}{ll}
 	x\pm1=2py_1^2\\
 	x\mp1=2qy_2^2,
 	\end{array}\right. \quad
 	\text{ or }\quad
 	(4):\  \left\{ \begin{array}{ll}
 	x\pm1=2y_1^2\\
 	x\mp1=2pqy_2^2,
 	\end{array}\right.
 	$$
 	Note that   $y=y_1y_2$  or  $y=2y_1y_2$.
 	\begin{enumerate}[\rm$\bullet$]
 		\item System $(1)$ implies 
 		$1=\left(\frac{y_1^2}{p}\right)=\left(\frac{x\pm 1}{p}\right)=\left(\frac{x\mp1\pm 2}{p}\right) =\left(\frac{2}{p}\right)=-1.$ So this case is impossible.
 		\item System $(3)$ implies 
 		$\left(\frac{2q}{p}\right)=\left(\frac{x\mp 1}{p}\right)=\left(\frac{\pm2}{p}\right)=\left(\frac{2}{p}\right)=-1$, which contradicts the fact that  $\left(\dfrac{p}{q}\right)=-1$.
 		\item If system $(4)$ holds, then  $2\varepsilon_{ pq}=2(x+y\sqrt{pq})=2(y_1+y_2\sqrt{pq})^2$. Thus $\varepsilon_{ pq}$ is a square in
 		$\mathbb{Q}(\sqrt{pq})$ which is absurd.
 		
 		\item Suppose that $\left\{ \begin{array}{ll}
 		x+1=py_1^2\\
 		x-1=qy_2^2,
 		\end{array}\right.$
 		then, 	$-1=\left(\frac{p}{q}\right)=\left(\frac{x+ 1}{q}\right)=\left(\frac{x-1+ 2}{q}\right)=\left(\frac{2}{q}\right)=1.$ Which is impossible too.
 	\end{enumerate}
 	From the above discussion, we infer that $\left\{ \begin{array}{ll}
 	x-1=py_1^2\\
 	x+1=qy_2^2.
 	\end{array}\right.$
 	So $ 2\varepsilon_{pq} =(\ y_1\sqrt{p}+ y_2\sqrt{q})^2$. Therefore $2\varepsilon_{pq}$ is a square in $\QQ(\sqrt{p}, \sqrt{q})$.
 	
 	Using similar techniques,  one can show that $ 2\varepsilon_{ q}$ is  a square in $\QQ(\sqrt{p}, \sqrt{q})$. Hence,
 	$\varepsilon_{pq}$ (resp. $\varepsilon_{ q}$) is not a square in  $\QQ(\sqrt{p}, \sqrt{q})$,  since otherwise we will get $\sqrt{2}\in \QQ(\sqrt{p}, \sqrt{q})$,  which is not true.
 	As $\varepsilon_{p}$ has  norm $-1$,  it follows that  $\sqrt{\varepsilon_{q}\varepsilon_{pq}}$ is the  only element of $\{\varepsilon_{p}^i\varepsilon_{q}^j\varepsilon_{pq}^k: i,  j \text{ and } k\in\{0, 1\}\}$ which is  a square in $\QQ(\sqrt{p}, \sqrt{q})$.
 	\;So the first item (cf. \cite{wada}). We similarly prove the second item.
 \end{proof}
 
 \begin{remark}\label{rmk1 on units}
 	By the previous proof,  we have $\varepsilon_{ pq}$ is a square in $\QQ(\sqrt 2,  \sqrt{p},  \sqrt{q})$. One can similarly show that if $q\equiv 3 \pmod 4$,  then $2{\varepsilon_{ q}}\ ($resp. $2{\varepsilon_{ 2q}})$  is a square in $\QQ(\sqrt{q})\  ($resp. $\QQ(\sqrt{2q}))$.
 \end{remark}

 \begin{lemma}\label{2}
 	Let $p$ and $q$ be two primes satisfying   conditions  {\eqref{cond 1}} or $(\ref{cond 2})$. Then the class number of $F=\mathbb{Q}(\sqrt{q},  \sqrt{p},  i)$ is odd.
 \end{lemma}
 \begin{proof}
 	For  values of class numbers of  quadratic number fields used below  one can see   \cite{connor88,  kaplan76}.
 	\begin{enumerate}[\rm $\bullet$]
 		\item If $p$ and $q$ satisfy  conditions  {\eqref{cond 1}}, then by class number formula (cf. \cite{wada}) we have
 		\begin{eqnarray*}
 			h_2(F)&=&\frac{1}{2^5}q(F)h_2(p)h_2(q)h_2(-p)h_2(-q)h_2(pq)h_2(-pq)h_2(-1),  \\
 			&=&\frac{1}{2^5}q(F)\cdot  1\cdot 1\cdot 2\cdot1 \cdot 2 \cdot 2 \cdot 1=\frac{1}{2^2}q(F).
 		\end{eqnarray*}
 		Since,   by Lemma \ref{lemma units of biquadratic fields},   a FSU of $F^+=\mathbb{Q}(\sqrt{q},  \sqrt{p})$ is given by    $\{\varepsilon_{p},   \varepsilon_{q},   \sqrt{\varepsilon_{q}\varepsilon_{pq}}\}$
 		and $\sqrt{2\varepsilon_{q}}$ is a square in $F^+$,   then by Lemma \ref{Lemme azizi},  $\{\varepsilon_{p},   \sqrt{\varepsilon_{q}\varepsilon_{pq}},    \sqrt{i\varepsilon_{q}}\}$
 		is a FSU of $F$. Thus $h_2(F)=\frac{1}{2^2}\cdot 4=1$.
 		\item If $p$ and $q$  satisfy   conditions $ {\eqref{cond 2}}$, then as above  we get
 		\begin{eqnarray*}
 			h_2(F)&=&\frac{1}{2^5}q(F)h_2(p)h_2(q)h_2(-p)h_2(-q)h_2(pq)h_2(-pq)h_2(-1),  \\
 			&=&\frac{1}{2^5}q(F)\cdot  1\cdot 1\cdot 1\cdot1 \cdot 1 \cdot 4 \cdot 1=\frac{1}{2^3}q(F).
 		\end{eqnarray*}
 		By Remark \ref{rmk1 on units}  we have  $ \sqrt{2\varepsilon_{q}}$,   $\sqrt{2\varepsilon_{p}}\in F^+$,   so   it is easy to see that a FSU of $F^+$ is
 		$\{\varepsilon_{q},   \sqrt{\varepsilon_{p}\varepsilon_{q}},   \sqrt{\varepsilon_{p}\varepsilon_{pq}} \text{ or } \sqrt{\varepsilon_{pq}}\}$ or
 		$\{\varepsilon_{q},    {\varepsilon_{pq}},   \sqrt{\varepsilon_{p}\varepsilon_{q}}\}$. Thus,  by Lemma \ref{Lemme azizi} and the  fact that $\frac{1}{2^3}q(F)\in \mathbb{N}$,   we have $q(F)=2^3$  (and thus  $\{\varepsilon_{q},    {\varepsilon_{pq}},   \sqrt{\varepsilon_{p}\varepsilon_{q}}\}$ is not a FSU of $F^+$).  Therefore   $h(F)$ is odd.
 	\end{enumerate}
 \end{proof}
 
 \begin{theorem}\label{thm 2-class group of the 3 triquad}
 	Let $p$ and $q$ be two primes    satisfying    conditions {\eqref{cond 1}} or $(\ref{cond 2})$. Then we have
 	\begin{enumerate}[\rm 1.]
 		\item The $2$-class group of $L_{pq}$ is of type $(2,  2)$.
 		\item The $2$-class group of $F_{pq}$ is of type $(2,  2)$.
 		\item The $2$-class group of $K_{pq}$ is cyclic of type $\ZZ/2^{m+1}\ZZ$,   	with $h_2(-2q)=2^m$.
 	\end{enumerate}
 	
 \end{theorem}
 \begin{proof} Assume firstly that $p$ and $q$ verify  conditions \eqref{cond 1}. So, by \cite{chemsZkhnin1},  the $2$-class group of $L_{pq}$ is of type $(2,  2)$. Under this assumption,   $h_2(p)=h_2(q)=h_2(-2)=1$,
 	$h_2(-2p)=h_2(pq)=2$ and $h_2(-2pq)=4$ (cf.\cite{connor88,  kaplan76}).
 	We claim that the class number of $k'=\QQ(\sqrt{p},   \sqrt{q})$ is odd. In fact,  we have
 	$h_2(p)=h_2(q) =1$ and
 	$ h_2(pq) =2$  (cf. \cite{connor88,  kaplan76}).
 	So Lemma \ref{lemma units of biquadratic fields} and Kuroda's class number formula (cf. \cite{lemmermeyer1994kuroda}) imply that
 	$h_2(k')=\frac{1}{4}q(k')h_2(p)h_2(q)h_2(pq)=1$.
 	Since there are only two primes
 	which ramify in  $K_{pq}/k'$,   then by  ambiguous class number formula  (cf. \cite{Gr}) the rank of the $2$-class group of $K_{pq}$ equals
 	$2-1-e=1-e$,  where $e$ is an integer defined by
 	$[E_{k'}: E_{k'}\cap N_{K_{pq}/k'}(K_{pq})]=2^e$. We infer that the rank of the  $2$-class group of $K_{pq}$ can not be equal to $2$.
 	On the other hand,   note that $L_{pq}$,   $F_{pq}$ and $K_{pq}$ are the three unramified quadratic extensions of $k_{pq}=\mathbb{Q}(\sqrt{-2},  \sqrt{pq})$ that having a $2$-class group  of  type $(2,  2)$ (cf. \cite[Théorème 1]{Be05}). Then,  by Remark \ref{rmk 1 preliminaries},  the $2$-class group  of $F_{pq}$ is of type $(2,  2)$ and that of $K_{pq}$ is cyclic. From \cite[Proposition 6]{Be05},  we deduce that $q(K_{pq})=4$. Hence,  class number formula  (cf. \cite{wada}) implies that
 	$$h_2(K_{pq})=\frac{1}{2^5}q(K_{pq})h_2(p)h_2(q)h_2(pq)h_2(-2p)h_2(-2q)h_2(-2pq)h_2(-2)
 	=2h_2(-2q).$$
 	So the result for   this case.
 	
 	Suppose now that $p$ and $q$ satisfy conditions  $(\ref{cond 2})$. 	Using  \cite[Proposition 5]{Be05} and  \cite{connor88,  kaplan76, lemmermeyer1994kuroda},  we  similarly show that $h_2(K_{pq})=2\cdot h_2(-2q)$. Thus,  $h_2(K_{pq})$ is divisible by $8$. Thus as above its $2$-class group can not be of type $(2, 2)$,  which completes the proof.
 \end{proof}
 
 \begin{corollary} Keep  assumptions  of Theorem $\ref{thm 2-class group of the 3 triquad}$, then  the group $G_{k_{pq}}$  is neither  abelian  nor quaternion of order $8$.
 \end{corollary}	
 \begin{proof}
 	Since $L_{pq}$,   $F_{pq}$ and $K_{pq}$ are the three unramified quadratic extensions of $k_{pq}$.
 	So we have the  result by Theorem \ref{thm 2-class group of the 3 triquad} and Remark \ref{rmk 1 preliminaries}.
 \end{proof}
 
 \begin{corollary}
 	Keep  assumptions of Theorem $\ref{thm 2-class group of the 3 triquad}$,  then the group $G_{K_{pq}}$ is cyclic of order $2^{m+1}$.
 \end{corollary}	
 
 \begin{theorem}\label{class group of L*}Let $p$ and $q$ be two primes satisfying    conditions{\eqref{cond 1}} or $ {\eqref{cond 2}} $. Then
 	the $2$-class group of  $L_{pq}^*=\mathbb{Q}(\sqrt2,  \sqrt{q},  \sqrt{p},  i)$ is $\ZZ/2^m\ZZ$,  	where $h_2(-2q)=2^m$.
 \end{theorem}
 \begin{proof}Let $F=\mathbb{Q}(\sqrt{q},  \sqrt{p},  i)$, so we know by  Lemma \ref{2} that the class number of $F$ is odd.
 	If  $p$ and $q$ verify   conditions {\eqref{cond 1}} (resp. $ {\eqref{cond 2}} $) we have $2$ is unramified in $k'=\QQ(\sqrt{p},   \sqrt{-q})$
 	(resp. $k''=\QQ(\sqrt{-p},   \sqrt{-q}))$,   then the decomposition group of $2$ is a non trivial cyclic subgroup of $Gal(k'/\QQ)$ (resp. $Gal(k''/\QQ)$) (in fact $2$ is inert in $\QQ(\sqrt{-pq} )$ (resp. $\QQ(\sqrt{-p})$ )). Since    a non trivial cyclic subgroup of $Gal(k'/\QQ)$ (resp. $Gal(k''/\QQ)$) has two elements then there are exactly $2$ primes of  $k'$ (resp. $k''$)  above $2$. As $2$ ramify in  $\QQ(\sqrt{-1})$,   it follows that there are exactly $2$ primes of $F$ above $2$.
 	Therefore,
 	there are exactly two primes that ramify in $L_{pq}^*/F$.
 	By  ambiguous class number formula (cf. \cite{Gr}) $rank (\mathrm{Cl}_2(L_{pq}^*))=2-1-e=1-e$, where 
 	$e$ is defined by $[E_{F}: E_{F}\cap N_{L_{pq}^*/F}(L_{pq}^*)]=2^e$.
 	Since $L_{pq}^*$ is the genus field of $L_{pq}=\mathbb{Q}(\sqrt2,  \sqrt{qp},  i)$,  so $[L_{pq}^*:L_{pq}]=2$; moreover  $\mathrm{Cl}_2(L_{pq})\simeq(2,  2)$,   then  the $2$-class group   $\mathrm{Cl}_2(L_{pq}^*)$ is cyclic or  of type $(2,  2)$ (cf. Remark \ref{rmk 1 preliminaries}). It follows  by the previous equality that
 	$\mathrm{Cl}_2(L_{pq}^*)$ is cyclic. Note that  $L_{pq}^*$ is an unramified quadratic extension of $K_{pq}$. So by Proposition \ref{prop the order of G_k} we have
 	${L_{pq}^*}^{(1)}=K_{pq}^{(2)}$ and $G_{K_{pq}}=2\cdot h_2(L_{pq}^{*})$.
 	Since by Theorem \ref{thm 2-class group of the 3 triquad} the $2$-class group of $K_{pq}$ is cyclic of order $2^{m+1}$,   it follows that ${L_{pq}^*}^{(1)}=K_{pq}^{(2)}=K_{pq}^{(1)}$ and $2\cdot h_2(L_{pq}^{*})=h_2(K_{pq})=2^{m+1}$. Hence we have the theorem.	
 \end{proof}
 
 \begin{corollary}
 	The group $G_{L_{pq}^*}$ is abelian.
 \end{corollary}
 
 \section{\textbf{Units of some multiquadratic number fields}}\label{section 1}	
 Let  $p$ and $q$ be   two prime integers satisfying   conditions ${\eqref{cond 1}}$, namely 
 $$p\equiv -q\equiv 1\pmod 4, 	\left(\dfrac{2}{p}\right)=-1,  \left(\dfrac{2}{q}\right)=1 \text{ and }  \left(\dfrac{p}{q}\right)=-1.$$
 Consider the field 	$\KK=\QQ(\sqrt 2,  \sqrt{p},  \sqrt{q},  \sqrt{-1})$,  and let $\KK^+=\QQ(\sqrt 2,  \sqrt{p},  \sqrt{q})$ be its maximal real subfield.
 The main task of this  section is to determine fundamental system of units of $\KK^+$ and $\KK$,  that will be used   to prove   Theorems  \ref{thm G_L quaternion} and \ref{th 2 on G_F}. 	
 To prove this result,  we have to  do some preparations. In the same manner as in the proof of Lemma \ref{lemma units of biquadratic fields}, one shows the following lemma.
 \begin{lemma}\label{lm expressions of units}
 	Let $p$ and $q$ be two primes satisfying    conditions $(\ref{cond 1})$.
 	\begin{enumerate}[\rm 1.] 		
 		\item  Let $\varepsilon_{pq}=a+b\sqrt{pq}$,    $a,\, b\in\ZZ$, then  $ p(a-1)$ is a square in $\NN$,  and   $\sqrt{2 \varepsilon_{pq}}=b_1\sqrt{p}+b_2\sqrt{q}$ and 	$2=-pb_1^2+qb_2^2$,  for some integers $b_1$ and $b_2$.		
 		\item Let $\varepsilon_{2pq}=x+y\sqrt{2pq}$, $x,\, y\in\ZZ$, then  $2p(x-1)$ is a square in $\NN$, and $\sqrt{2\varepsilon_{2pq}}=y_1\sqrt{2p}+y_2\sqrt{q}$ and 	$2=-2py_1^2+qy_2^2$,  for some integers $y_1$ and $y_2$.
 		\item  Let $ \varepsilon_{2q}=c +d\sqrt{2q}$,   $c,\, d\in\ZZ$, then     $c+1$ is a square in $\NN$, and 
 		$\sqrt{ 2\varepsilon_{  2q}}=d_1 +d_2\sqrt{2q}$ and 	$2=d_1^2-2qd_2^2$,  for some integers $d_1$ and $d_2$.
 		\item  Let  $ \varepsilon_{q}=c' +d'\sqrt{q}$,  $c', \, d'\in\ZZ$, then    $c'+1$ is a square in $\NN$, and
 		$\sqrt{ 2\varepsilon_{  q}}=d_1' +d_2'\sqrt{q}$ and 	$2=d_1'^2-qd_2'^2$,  for some integers $d_1'$ and $ d_2'$.		
 	\end{enumerate}
 \end{lemma}	
 
 \begin{lemma}\label{29}
 	Let $p$ and $q$ be two primes satisfying  conditions   ${\eqref{cond 1}}$.
 	Let $\KK^+=\QQ(\sqrt 2,  \sqrt{p},  \sqrt{q})$, so the unit group of $\KK^+$ is one of the following:
 	\begin{enumerate}[\rm 1.]
 		\item $E_{\mathbb{K}^+}=\langle-1,  \varepsilon_{2},  \varepsilon_{p},  \sqrt{\varepsilon_{q}},  \sqrt{\varepsilon_{2q}},  \sqrt{\varepsilon_{pq}},  \sqrt{\varepsilon_{2pq}},  \sqrt{\varepsilon_{2}\varepsilon_{p}\varepsilon_{2p}}\rangle, $
 		\item $E_{\mathbb{K}^+}=\langle-1,  \varepsilon_{2},  \varepsilon_{p},  \sqrt{\varepsilon_{q}},  \sqrt{\varepsilon_{2q}},  \sqrt{\varepsilon_{pq}},   \sqrt{\varepsilon_{2}\varepsilon_{p}\varepsilon_{2p}},
 		\sqrt[4]{\varepsilon_{2}^2\varepsilon_{q}\varepsilon_{pq}\varepsilon_{2pq}}\rangle.$
 	\end{enumerate}
 \end{lemma}
 \begin{proof}
 	To prove this lemma,  we use the algorithm described by \cite{wada}. Consider the  following   diagram (see Figure \ref{fig:1}):
 	\begin{figure}[H]
 		$$\xymatrix@R=0.8cm@C=0.3cm{
 			&\KK^+=\QQ( \sqrt 2,  \sqrt{p},  \sqrt{q})\ar@{<-}[d] \ar@{<-}[dr] \ar@{<-}[ld] \\
 			L_1=\QQ(\sqrt 2, \sqrt{p})\ar@{<-}[dr]& L_2=\QQ(\sqrt 2,  \sqrt{q}) \ar@{<-}[d]& L_3=\QQ(\sqrt 2,  \sqrt{pq})\ar@{<-}[ld]\\
 			&\QQ(\sqrt 2)}$$
 		\caption{Subfields of $\KK^+/\QQ(\sqrt 2)$}\label{fig:1}
 	\end{figure}
 	
 	\noindent By \cite{azizitalbi},  Lemma \ref{lemma units of biquadratic fields} and \cite{chemsZkhnin1} we have
 	a FSU of $L_1$ is given by $\{\varepsilon_{2},  \varepsilon_{p},
 	\sqrt{\varepsilon_{2}\varepsilon_{p}\varepsilon_{2p}}\}$,  a FSU of $L_2$ is given by $\{\varepsilon_{2},  \sqrt{\varepsilon_{q}},  \sqrt{\varepsilon_{2q}}\}$ and   a FSU of $L_3$ is given by $\{ \varepsilon_{2},  \varepsilon_{pq},
 	\sqrt{\varepsilon_{pq}\varepsilon_{2pq}}\}$. 	
 	It follows that,   	$$E_{L_1}E_{L_2}E_{L_3}=\langle-1,   \varepsilon_{2},  \varepsilon_{p},  \varepsilon_{pq},  \sqrt{\varepsilon_{q}},  \sqrt{\varepsilon_{2q}},   \sqrt{\varepsilon_{pq}\varepsilon_{2pq}},  \sqrt{\varepsilon_{2}\varepsilon_{p}\varepsilon_{2p}}\rangle.$$	
 	Note that a FSU of $\KK$ consists  of seven  units chosen from those of $L_1$,  $L_2$ and $L_3$,  and  from the square  roots of the units of $E_{L_1}E_{L_2}E_{L_3}$ which are squares in $\KK$ (cf. \cite{wada}).
 	Thus we shall determine elements of $E_{L_1}E_{L_2}E_{L_3}$ which are squares in $\KK^+$. Suppose that  $X$ is an element of $\KK^+$ which is a square of an element of $E_{L_1}E_{L_2}E_{L_3}$, so
 	$$X^2=\varepsilon_{2}^a\varepsilon_{p}^b \varepsilon_{pq}^c\sqrt{\varepsilon_{q}}^d\sqrt{\varepsilon_{2q}}^e\sqrt{\varepsilon_{pq}\varepsilon_{2pq}}^f
 	\sqrt{\varepsilon_{2}\varepsilon_{p}\varepsilon_{2p}}^g, $$
 	where $a,  b,  c,  d,  e,  f$ and $g$ are in $\{0,  1\}$.
 	
 	We shall use norm maps from $\KK^+$ to its  sub-extensions  to eliminate  the cases of $X^2$ which do not occur. Set
 	$G=\mathrm{Gal}(\KK^+/\QQ)=\langle \sigma_1,  \sigma_2,  \sigma_3\rangle$,
 	where\begin{center} $\sigma_1(\sqrt{2})=-\sqrt{2}$,  $\sigma_2(\sqrt{p})=-\sqrt{p}$ and $\sigma_3(\sqrt{q})=-\sqrt{q}$, \\ and
 		$\sigma_i(\sqrt{2})=\sqrt{2}$ for $i\in\{2,  3\}$, \\
 		$\sigma_i(\sqrt{p})=\sqrt{p}$ for $i\in\{1,  3\}$ and \\
 		$\sigma_i(\sqrt{q})=\sqrt{q}$ for $i\in\{1,  2\}$.\end{center} Hence $L_1$,  $L_2$ and $L_3$ are
 	the fixed fields of the subgroups of $G$ generated respectively by $\sigma_3$,  $\sigma_2$ and $\sigma_2\sigma_3$.

 	Let us firstly do some computations that will help in the computations of these norm maps which we shall use now and for the proof of the next lemmas as well.
 	Let $L_4=\QQ(\sqrt p,  \sqrt q)$ and $L_5=\QQ(\sqrt p,  \sqrt{2q})$. By Lemma \ref{lemma units of biquadratic fields}, 
 	a FSU of $L_4$ is  $\{\varepsilon_{p},  \varepsilon_{q},  \sqrt{\varepsilon_{q}\varepsilon_{pq}}\}$,
 	and  a FSU of $L_5$ is  $\{\varepsilon_{p},  \varepsilon_{2q},  \sqrt{\varepsilon_{2pq}}\}$.

 	Let $u,  v,  t,  k$ and $ r\in \{0,  1\}$. Table  \ref{table1} (page \pageref{table1}) will be used to compute norm maps. Its  $8$-th line,  for example,   is constructed as follows (the other lines are constructed in the same manner).  By Lemma \ref{lm expressions of units},    there are two  integers $y_1$ and $y_2$ such that
 	$\sqrt{\varepsilon_{2pq}}=\frac{\sqrt{2}}{2}(y_1\sqrt{2p}+y_2\sqrt{q})$. Which implies that
 	$$\begin{array}{ll}
 	\sqrt{\varepsilon_{2pq}}^{\sigma_1}&=\frac{-\sqrt{2}}{2}(-y_1\sqrt{2p}+y_2\sqrt{q})\\
 	&=\frac{\frac{-\sqrt{2}}{2}(-y_1\sqrt{2p}+y_2\sqrt{q})(y_1\sqrt{2p}+y_2\sqrt{q})}{(y_1\sqrt{2p}+y_2\sqrt{q})}\\
 	&=\frac{\frac{-\sqrt{2}}{2}(-y_1^22p+y_2^2q)}{\sqrt{2}\sqrt{\varepsilon_{2pq}}}\\
 	&=\frac{\frac{-\sqrt{2}}{2}\cdot2}{\sqrt{2}\sqrt{\varepsilon_{2pq}}}
 	=\frac{-1}{\sqrt{\varepsilon_{2pq}}}\cdot
 	\end{array}$$
 	\begin{center}$\begin{array}{ll}
 		\sqrt{\varepsilon_{2pq}}^{\sigma_2}&=\frac{\sqrt{2}}{2}(-y_1\sqrt{2p}+y_2\sqrt{q})\\
 		&=-\frac{-\sqrt{2}}{2}(-y_1\sqrt{2p}+y_2\sqrt{q})\\
 		&=-\frac{-1}{\sqrt{\varepsilon_{2pq}}}
 		=\frac{1}{\sqrt{\varepsilon_{2pq}}}\cdot
 		\end{array}$
 		$\begin{array}{ll}
 		\sqrt{\varepsilon_{2pq}}^{\sigma_3}&=\frac{\sqrt{2}}{2}(y_1\sqrt{2p}-y_2\sqrt{q})\\
 		&=\frac{-\sqrt{2}}{2}(-y_1\sqrt{2p}+y_2\sqrt{q})\\
 		&=\frac{-1}{\sqrt{\varepsilon_{2pq}}}\cdot \end{array}$\end{center}
 	
 	$$\begin{array}{ll}
 	\sqrt{\varepsilon_{2pq}}^{1+\sigma_1}&=\sqrt{\varepsilon_{2pq}}\cdot \frac{-1}{\sqrt{\varepsilon_{2pq}}}=-1.\\ \sqrt{\varepsilon_{2pq}}^{1+\sigma_2}&=\sqrt{\varepsilon_{2pq}}\cdot \frac{1}{\sqrt{\varepsilon_{2pq}}}=1.\\ \sqrt{\varepsilon_{2pq}}^{1+\sigma_3}&=\sqrt{\varepsilon_{2pq}}\cdot \frac{-1}{\sqrt{\varepsilon_{2pq}}}=-1.\\
 	\sqrt{\varepsilon_{2pq}}^{1+\sigma_1\sigma_2}&=\sqrt{\varepsilon_{2pq}}\cdot (\frac{1}{\sqrt{\varepsilon_{2pq}}})^{\sigma_{1}}=-\varepsilon_{2pq}.\\
 	\sqrt{\varepsilon_{2pq}}^{1+\sigma_1\sigma_3}&=\sqrt{\varepsilon_{2pq}}\cdot (\frac{-1}{\sqrt{\varepsilon_{2pq}}})^{\sigma_{1}}=\varepsilon_{2pq}.\\
 	\sqrt{\varepsilon_{2pq}}^{1+\sigma_2\sigma_3}&=\sqrt{\varepsilon_{2pq}}\cdot (\frac{-1}{\sqrt{\varepsilon_{2pq}}})^{\sigma_{2}}=-\varepsilon_{2pq}.
 	\end{array}$$
 	
 	Now we return back to our square
 	$X^2=\varepsilon_{2}^a\varepsilon_{p}^b \varepsilon_{pq}^c\sqrt{\varepsilon_{q}}^d\sqrt{\varepsilon_{2q}}^e\sqrt{\varepsilon_{pq}\varepsilon_{2pq}}^f
 	\sqrt{\varepsilon_{2}\varepsilon_{p}\varepsilon_{2p}}^g$, 
 	by applying   the norm $N_{\KK^+/L_2}=1+\sigma_2$ (see Table  \ref{table1} page \pageref{table1}) we get:
 	\begin{eqnarray*}
 		N_{\KK^+/L_2}(X^2)&=&\varepsilon_{2}^{2a}\cdot (-1)^{b}\cdot 1\cdot \varepsilon_q^d \cdot \varepsilon_{2q}^e \cdot 1 \cdot (-1)^{ vg} \cdot \varepsilon_{2}^g \\
 		&=&\varepsilon_{2}^{2a}\varepsilon_{q}^{d}\varepsilon_{2q}^{e}(-1)^{b+vg}\varepsilon_{2}^{g}.
 	\end{eqnarray*}
 	As $\varepsilon_{q}$,   $\varepsilon_{2q}$ are squares in $L_2$ and $N_{\KK^+/L_2}(X^2)>0$,   so  $b+vg\equiv0\pmod2$ and $\varepsilon_{2}^g$ is a square in $L_2$. But $\varepsilon_{2}$ is not  a square in $L_2$,   then $g=0$ and thus $b=0$. Therefore, $X^2$ become 
 	$$X^2=\varepsilon_{2}^a \varepsilon_{pq}^c\sqrt{\varepsilon_{q}}^d\sqrt{\varepsilon_{2q}}^e\sqrt{\varepsilon_{pq}\varepsilon_{2pq}}^f.$$
 	Similarly,   by  applying   $N_{\KK^+/L_3}=1+\sigma_2\sigma_3$ (see Table  \ref{table1} page \pageref{table1}) one gets:
 	$$N_{\KK^+/L_3}(X^2)=\varepsilon_{2}^{2a}\cdot \varepsilon_{pq}^{2c} \cdot (\varepsilon_{pq}\varepsilon_{2pq})^f,  $$
 	unfortunately,   here   we conclude nothing. So we will use the norm map over $L_4=\QQ(\sqrt p,   \sqrt q)$ which is $N_{\KK^+/L_4}=1+\sigma_1$.
 	Note that $\{\varepsilon_{p},   \varepsilon_{q},   \sqrt{\varepsilon_{q}\varepsilon_{pq}}\}$  is a FSU of $L_4$. Thus
 	\begin{eqnarray*}
 		N_{\KK^+/L_4}(X^2)&=&(-1)^{a}\cdot\varepsilon_{pq}^{2c}\cdot (-\varepsilon_{q})^d\cdot (-1)^e \cdot (\varepsilon_{pq})^f\\
 		&=&\varepsilon_{pq}^{2c} \cdot (-1)^{a+d+e}\cdot
 		\varepsilon_{q}^d\cdot \varepsilon_{pq}^f>0.
 	\end{eqnarray*}
 	
 	\noindent Since $N_{\KK^+/L_4}(X^2)>0$,   then $a+d+e\equiv0\pmod2$.  By Remark \ref{rmk1 on units},  $2\varepsilon_{q}$ is a square in $\QQ(\sqrt q)$ and $2\varepsilon_{pq}$ is a square in $\QQ(\sqrt p,   \sqrt q)$.
 	So $d=f$,  since otherwise we will get   $\varepsilon_{pq}$ or $\varepsilon_{q}$ is a square in $L_4$. Therefore,
 	$$X^2=\varepsilon_{2}^a \varepsilon_{pq}^c\sqrt{\varepsilon_{q}}^d\sqrt{\varepsilon_{2q}}^e\sqrt{\varepsilon_{pq}\varepsilon_{2pq}}^d.$$
 	Now we     apply $N_{\KK^+/L_5}=1+\sigma_1\sigma_3$,   where $L_5=\QQ(\sqrt p,   \sqrt{2q})$. Note that   a FSU of $L_5$  is given by $\{\varepsilon_{p},   \varepsilon_{2q},   \sqrt{\varepsilon_{2pq}}\}$.
 	We have
 	\begin{eqnarray*}
 		N_{\KK^+/L_5}(X^2)&=&(-1)^{a}\cdot 1 \cdot  (-1)^{d} \cdot (-\varepsilon_{2q})^{e} \cdot \varepsilon_{2pq}^d\\
 		&=&\varepsilon_{2pq}^d \cdot (-1)^{a+d+e} \cdot \varepsilon_{2q}^e>0.
 	\end{eqnarray*}
 	So  $a+d+e\equiv0\pmod2$. As $\varepsilon_{2q}$ is not a square in $L_5$,  then $e=0$ and  $a+d\equiv0\pmod2$. Therefore $a=d$ and
 	$$X^2=\varepsilon_{2}^d \varepsilon_{pq}^c\sqrt{\varepsilon_{q}}^d\sqrt{\varepsilon_{pq}\varepsilon_{2pq}}^d.$$
 	Remark that $\varepsilon_{pq}$ is a square in $\KK^+$,   so we may put
 	$$X^2=\varepsilon_{2}^d\sqrt{\varepsilon_{q}}^d\sqrt{\varepsilon_{pq}\varepsilon_{2pq}}^d.$$
 	Applying the norm  maps from $\KK^+$ to all the rest of its sub-extensions, no contradiction is obtained and we conclude nothing about $d$. So $d=0$ or $1$, and thus the result.	
 \end{proof}

 \newpage
 {\begin{table}[H]
 		\renewcommand{\arraystretch}{2.5}
 		
 		\begin{center}\rotatebox{-90}{
 				\begin{tabular}{|c|c|c|c|c|c|c|c|c|c|}
 					\hline
 					$\varepsilon$ & $\varepsilon^{\sigma_1}$ & $\varepsilon^{\sigma_2}$ & $\varepsilon^{\sigma_3}$ &$\varepsilon^{1+\sigma_1}$ & $\varepsilon^{1+\sigma_2}$ & $\varepsilon^{1+\sigma_3}$ & $\varepsilon^{1+\sigma_1\sigma_2}$& $\varepsilon^{1+\sigma_1\sigma_3}$& $\varepsilon^{1+\sigma_2\sigma_3}$  \\ \hline
 					$\varepsilon_{2}$ & $\frac{-1}{\varepsilon_{2}}$ & $\varepsilon_{2}$ & $\varepsilon_{2}$ & $-1$ & $\varepsilon_{2}^2$ &$\varepsilon_{2}^2$&$-1$ & $-1$ & $\varepsilon_{2}^2$\\ \hline
 					
 					$\varepsilon_{p}$ & $\varepsilon_{p}$ & $\frac{-1}{\varepsilon_{p}}$ & $\varepsilon_{p}$ &$\varepsilon_{p}^2$ &$-1$ & $\varepsilon_{p}^2$ &$-1$& $\varepsilon_{p}^2$ &$-1$\\ \hline
 					$\varepsilon_{pq}$ & $\varepsilon_{pq}$ & $\frac{1}{\varepsilon_{pq}}$ & $\frac{1}{\varepsilon_{pq}}$ & $\varepsilon_{pq}^2$& $1$ & $1$& $1$& $1$ &$\varepsilon_{pq}^2$\\ \hline
 					$\sqrt{\varepsilon_{q}}$ & $-\sqrt{\varepsilon_{q}}$ & $\sqrt{\varepsilon_{q}}$ & $\frac{1}{\sqrt{\varepsilon_{q}}}$ & $-\varepsilon_{q}$ & $\varepsilon_{q}$ &$1$&$-\varepsilon_{q}$&$-1$&$1$\\ \hline
 					$\sqrt{\varepsilon_{2q}}$ & $\frac{-1}{\sqrt{\varepsilon_{2q}}}$ & $\sqrt{\varepsilon_{2q}}$ & $\frac{1}{\sqrt{\varepsilon_{2q}}}$ &$-1$ & $\varepsilon_{2q}$ & $1$ &$-1$& $-\varepsilon_{2q}$& $1$\\ \hline
 					
 					$\sqrt{\varepsilon_{pq}}$ &$-\sqrt{\varepsilon_{pq}}$&$\frac{1}{\sqrt{\varepsilon_{pq}}}$ &$\frac{-1}{\sqrt{\varepsilon_{pq}}}$& $-{\varepsilon_{pq}}$ &$1$ & $-1$ & $-1$& $1$ & $-\varepsilon_{pq}$ \\ \hline
 					
 					$\sqrt{\varepsilon_{2pq}}$ &$\frac{-1}{\sqrt{\varepsilon_{2pq}}}$&$\frac{1}{\sqrt{\varepsilon_{2pq}}}$ &$\frac{-1}{\sqrt{\varepsilon_{2pq}}}$& $-1$ &$1$ & $-1$ & $-\varepsilon_{2pq}$& $\varepsilon_{2pq}$ & $-\varepsilon_{2pq}$ \\ \hline
 					
 					$\sqrt{\varepsilon_{2}\varepsilon_{p}\varepsilon_{2p}}$ & $(-1)^u\sqrt{\frac{\varepsilon_{p}}{\varepsilon_{2}\varepsilon_{2p}}}$ &  $(-1)^v\sqrt{\frac{\varepsilon_{2}}{\varepsilon_{p}\varepsilon_{2p}}}$ &
 					$\sqrt{\varepsilon_{2}\varepsilon_{p}\varepsilon_{2p}}$ & $(-1)^u\varepsilon_{p}$&$(-1)^v\varepsilon_{2}$ & $\varepsilon_{2}\varepsilon_{p}\varepsilon_{2p}$ &  && \\
 					\hline

 					$\sqrt[4]{\varepsilon_{2}^2\varepsilon_{q}\varepsilon_{pq}\varepsilon_{2pq}}$ & $(-1)^k\sqrt[4]{\frac{\varepsilon_{q}\varepsilon_{pq}}{\varepsilon_{2}^2\varepsilon_{2pq}}}$ &  $(-1)^t\sqrt[4]{\frac{\varepsilon_{2}^2\varepsilon_{q}}{\varepsilon_{pq}\varepsilon_{2pq}}}$ &
 					$(-1)^r\sqrt[4]{\frac{\varepsilon_{2}^2}{\varepsilon_{q}\varepsilon_{pq}\varepsilon_{2pq}}}$ & $(-1)^k\sqrt{\varepsilon_{q}\varepsilon_{pq}}$&$(-1)^t\varepsilon_{2}\sqrt{\varepsilon_{q}}$ &  $(-1)^r\varepsilon_{2}$ &  && \\
 					\hline
 					
 				\end{tabular}
 			}\caption{Image of units by $\sigma_i$}\label{table1}	\end{center}
 \end{table} }
 
 \begin{lemma}\label{thm units case T}
 	Suppose that the unit group of  $\KK^+$ takes the form in the second item of   Lemma $\ref{29}$. Then
 	the unit group of $\KK$ is one of the following.
 	\begin{enumerate}[\rm1.]
 		\item  $E_{\mathbb{K} }=\langle \zeta_8,  \varepsilon_{2},  \varepsilon_{p},  \sqrt{\varepsilon_{q}},  \sqrt{\varepsilon_{2q}},  \sqrt{\varepsilon_{pq}},   \sqrt{\varepsilon_{2}\varepsilon_{p}\varepsilon_{2p}},
 		\sqrt[4]{\varepsilon_{2}^2\varepsilon_{q}\varepsilon_{pq}\varepsilon_{2pq}}\rangle$,  or
 		\item $E_{\mathbb{K} }=\langle   \zeta_8,  \varepsilon_{2},  \varepsilon_{p},  \sqrt{\varepsilon_{q}} ,  \sqrt{\varepsilon_{pq}},   \sqrt{\varepsilon_{2}\varepsilon_{p}\varepsilon_{2p}},
 		\sqrt[4]{\varepsilon_{2}^2\varepsilon_{q}\varepsilon_{pq}\varepsilon_{2pq}},   \sqrt[4]{ \zeta_8^2\varepsilon_{q}\varepsilon_{2 q}}\rangle.$
 	\end{enumerate}
 \end{lemma}
 \begin{proof}
 	We shall make use of  Table  \ref{table1} page \pageref{table1},  and  respect the same notations of the proof of the previous Lemma \ref{29}.	According to Lemma \ref{Lemme azizi}, set
 	$$Y^2=(2+\sqrt{2})\varepsilon_{2}^a\varepsilon_{p}^b\sqrt{\varepsilon_{q}}^c\sqrt{\varepsilon_{2q}}^d\sqrt{\varepsilon_{pq}}^e
 	\sqrt{\varepsilon_{2}\varepsilon_{p}\varepsilon_{2p}}^f\sqrt[4]{\varepsilon_{2}^2\varepsilon_{q}\varepsilon_{pq}\varepsilon_{2pq}}^g.$$
 	We have $N_{\KK^+/L_2}=1+\sigma_2$. So
 	\begin{eqnarray*}
 		N_{\KK^+/L_2}(Y^2)&=&(2+\sqrt{2})^2\cdot  \varepsilon_{2}^{2a}\cdot (-1)^b\cdot  \varepsilon_{q}^c \cdot \varepsilon_{2q}^d\cdot 1\cdot (-1)^{fv}\cdot  \varepsilon_{2}^f\cdot  (-1)^{gt}\cdot  \varepsilon_{2}^{g}\cdot  \sqrt{\varepsilon_{q}}^g,  \\	
 		&=&(2+\sqrt{2})^2\cdot\varepsilon_{2}^{2a}\cdot \varepsilon_{q}^c\cdot\varepsilon_{2q}^d\cdot (-1)^{fv+b+gt}\cdot \varepsilon_{2}^{g+f}\cdot \sqrt{\varepsilon_{q}}^g>0,
 	\end{eqnarray*}
 	We have $fv+b+gt=0\pmod 2$. Recall that  a $\mathrm{FSU}$  of $L_2$ is $\{\varepsilon_{2},   \sqrt{\varepsilon_{q}},   \sqrt{\varepsilon_{2q}}\}$, so 
 	\begin{enumerate}[\rm $\bullet$]
 		\item the case $g=0$ and $f=1$ is impossible. In fact $\sqrt{\varepsilon_{2}}\not\in L_2$,
 		\item the case $g=1$ and $f=0$ is impossible too. In fact $\sqrt{\varepsilon_{2}\sqrt{\varepsilon_{q}}}\not\in L_2$,
 		\item the case $g=1$ and $f=1$ is impossible too. In fact $\sqrt[4]{ \varepsilon_{q}}\not\in L_2$.
 	\end{enumerate}
 	It follows that   $f=g=0 $ and  $b=0$.	Thus,  	
 	$$Y^2=(2+\sqrt{2})\varepsilon_{2}^a \sqrt{\varepsilon_{q}}^c\sqrt{\varepsilon_{2q}}^d\sqrt{\varepsilon_{pq}}^e.$$
 	We have $N_{\KK^+/L_4}=1+\sigma_1$. So	
 	\begin{eqnarray*}
 		N_{\KK^+/L_4}(Y^2)&=&(4-2)\cdot {(-1)}^{a} \cdot {(-1)}^c\cdot      \varepsilon_q^c \cdot {(-1)}^d\cdot      ({-1})^e\cdot \varepsilon_{pq}^e,  \\	
 		&=& (-1)^{a+c+d+e} \cdot 2\cdot \varepsilon_q^c \varepsilon_{pq}^e  >0.
 	\end{eqnarray*}
 	We have $a+c+d+e=0\pmod2$. Note that $\{\varepsilon_{p},   \varepsilon_{q},   \sqrt{\varepsilon_{q}\varepsilon_{pq}}\}$  is a FSU of $L_4$. Since $2$ is not a square in $L_4$,  then $e\not=c$. It follows that $a+d+1=0\pmod2$. Therefore,  $a\not=d$ and $e\not=c$. To summarize we have
 	$$Y^2=(2+\sqrt{2})\varepsilon_{2}^a \sqrt{\varepsilon_{q}}^c\sqrt{\varepsilon_{2q}}^d\sqrt{\varepsilon_{pq}}^e, $$
 	with $a\not=d$ and $e\not=c$.
 	Let us now apply $	N_{\KK^+/L_3}=1+\sigma_{2}\sigma_{3}$. So
 	\begin{eqnarray*}
 		N_{\KK^+/L_3}(Y^2)&=&(2+\sqrt{2})^2 \cdot \varepsilon_{2}^{2a}\cdot 1\cdot 1\cdot (-1)^{e} \cdot \sqrt{\varepsilon_{pq}}^e>0.  \\	
 	\end{eqnarray*}
 	Thus,  $e=0$ and so $c=1$. Then we have
 	$$Y^2=(2+\sqrt{2})\varepsilon_{2}^a \sqrt{\varepsilon_{q}}\sqrt{\varepsilon_{2q}}^d.$$
 	By applying $	N_{\KK^+/L_5}=1+\sigma_{1}\sigma_{3}$ with $\mathbb{Q}(\sqrt{p}, \sqrt{2q})$,  we get
 	\begin{eqnarray*}
 		N_{\KK^+/L_5}(Y^2)&=&(4-2)\cdot (-1)^a\cdot (-1) \cdot    (-1)^d \cdot\varepsilon_{2q}^{d}\\
 		&= &  (-1)^{a+d+1}\cdot 2\cdot \varepsilon_{2q}^d>0.  	
 	\end{eqnarray*}
 	Since $2$ is not a square in $L_5$,  then $d=1$ and so $a=0$. Thus  $Y^2=(2+\sqrt{2}) \sqrt{\varepsilon_{q}}\sqrt{\varepsilon_{2q}}$. So the results are gotten  by applying Lemma \ref{Lemme azizi}. 	
 \end{proof}
 
 \noindent Now we are able to state and  prove the main theorem of this section.
 \begin{theorem}  \label{thm main theorem on units}
 	Let $p$ and $q$ be two primes satisfying   conditions   ${\eqref{cond 1}}$. Let $\KK=\QQ(\sqrt 2,  \sqrt{p},  \sqrt{q},  \sqrt{-1})$  and
 	$\KK^+=\QQ(\sqrt 2,  \sqrt{p},  \sqrt{q})$. Then we have:
 	\begin{enumerate}[\rm 1.]
 		\item $E_{\mathbb{K}^+}=\langle-1,  \varepsilon_{2},  \varepsilon_{p},  \sqrt{\varepsilon_{q}},  \sqrt{\varepsilon_{2q}},  \sqrt{\varepsilon_{pq}},   \sqrt{\varepsilon_{2}\varepsilon_{p}\varepsilon_{2p}},
 		\sqrt[4]{\varepsilon_{2}^2\varepsilon_{q}\varepsilon_{pq}\varepsilon_{2pq}}\rangle.$		
 		\item $E_{\mathbb{K} }=\langle   \zeta_8,  \varepsilon_{2},  \varepsilon_{p},  \sqrt{\varepsilon_{q}} ,  \sqrt{\varepsilon_{pq}},   \sqrt{\varepsilon_{2}\varepsilon_{p}\varepsilon_{2p}},
 		\sqrt[4]{\varepsilon_{2}^2\varepsilon_{q}\varepsilon_{pq}\varepsilon_{2pq}},   \sqrt[4]{ \zeta_8^2\varepsilon_{q}\varepsilon_{2 q}}\rangle.$				
 	\end{enumerate}		
 \end{theorem}
 \begin{proof}	
 	Under conditions   ${\eqref{cond 1}}$,	we have $h_2(p)=h_2(q)=h_2(-q)=h_2(2q)=h_2(-1)=h_2(2)=h_2(-2)=1$,
 	$h_2(-p)=h_2(2p)=h_2(-2p)=h_2(pq)=h_2(-pq)=h_2(2pq)=2$ and $h_2(-2pq)=4$ (cf. \cite{connor88,  kaplan76}). So by Wada's class number formula (cf. \cite{wada}) one gets
 	$$\begin{array}{ll}
 	h_2(\KK)&=\frac{1}{2^{16}}q(\KK) h_2(-1) h_2(2) h_2(-2) h_2(p) h_2(-p) h_2(q)\\
 	&\qquad h_2(-q) h_2(2p)h_2(-2p) h_2(2q)h_2(-2q) h_2(pq) h_2(-pq) h_2(2pq) h_2(-2pq)\\
 	&=\frac{1}{2^{16}}\cdot q(\KK)\cdot 1\cdot 1 \cdot 1 \cdot 1 \cdot 2 \cdot 1 \cdot 1 \cdot 2 \cdot 2\cdot 1\cdot h_2(-2q)\cdot 2\cdot 2\cdot 2\cdot 4,  \\
 	&= \frac{1}{2^{8}}.q(\KK).h_2(-2q).
 	\end{array}$$
 	On the other hand, by Theorem \ref{class group of L*} we have  $h_2(\KK)=h_2(-2q)$. So obviously we must have  $q(\KK)= 2^{8}$.
 	\begin{enumerate}[\rm $\bullet$]
 		\item Suppose that the unit group of $\KK^+ $ takes the form in the first item of  Lemma \ref{29},  then a FSU of $\KK^+ $ is
 		$$ \{ \varepsilon_{2},  \varepsilon_{p},  \sqrt{\varepsilon_{q}},  \sqrt{\varepsilon_{2q}},  \sqrt{\varepsilon_{pq}},  \sqrt{\varepsilon_{2pq}},  \sqrt{\varepsilon_{2}\varepsilon_{p}\varepsilon_{2p}} \} =\{\alpha_1, \alpha_2, ..., \alpha_7 \}.$$ 
 		Thus by  Lemma \ref{Lemme azizi} a FSU of $\KK$ is $\{\alpha_1, \alpha_2, ..., \alpha_7 \}$  or $\{ \alpha_{i_1}, ..., \alpha_{i_5},  \alpha_{i_0},  \sqrt{\zeta_8\alpha} \}$ with  $i_k\in \{1, ..., 7\}$ and
 		$\alpha=\alpha_1^{r_1}\alpha_2^{r_2}\cdots\alpha_7^{r_7}$,  where $r_k\in\{0, 1\} $, and $\alpha_{i_0}\in \{\varepsilon_{2}, \varepsilon_{p}\}$,  for some $i_0$. Thus  $q(\KK)\leq 2^{7}$. Which is absurd.
 		\item Assume now that  the unit group of  $\KK$ takes the form in the first item of  Lemma \ref{thm units case T},  then $q(\KK)\leq 2^{7}$. which is also absurd.
 	\end{enumerate}
 	Thus the only possible case is the one which is given by the second item of  Lemma \ref{thm units case T}. Which completes the proof.
 \end{proof}
 
 \section{$2$-class field towers of some multiquadratic number fields}\label{section 2}~\\
 Keep the notations of the previous sections. Now we can investigate  the structure of the second $2$-class groups of $L_{pq}$ and $F_{pq}$ (i.e. $G_{L_{pq}}$ and $G_{F_{pq}}$) defined in section \ref{section 01}.
 
 \begin{lemma}\label{norms  case T} Let $p$ and $q$ be two primes satisfying    conditions {\eqref{cond 1}}. Then
 	\begin{enumerate}[$\bullet$]
 		\item    $N_{L_{pq}^*/L_{pq}}(\sqrt{\varepsilon_{q}})= 1$ and    $N_{L_{pq}^*/L_{pq}}(\sqrt{\varepsilon_{pq}})=-\varepsilon_{pq}$.
 		\item   $N_{L_{pq}^*/L_{pq}}(\varepsilon_{2})=\varepsilon_{2}^2$,  $N_{L_{pq}^*/L_{pq}}(\zeta_8)=i$ and $N_{L_{pq}^*/L_{pq}}(\varepsilon_{p})=-1$.
 		\item $N_{L_{pq}^*/L_{pq}}(\sqrt{\varepsilon_{2}\varepsilon_{p}\varepsilon_{2p}})=\pm \varepsilon_{2}$,  $N_{L_{pq}^*/L_{pq}}(\sqrt[4]{\zeta_8^2	\varepsilon_{q}\varepsilon_{2q}})=\pm\zeta_8 $ and   $N_{L_{pq}^*/L_{pq}}(\sqrt[4]{\varepsilon_{2}^2{\varepsilon_{q}}{\varepsilon_{pq}\varepsilon_{2pq}}})=\pm \varepsilon_{2}\sqrt{\varepsilon_{pq}\varepsilon_{2pq}}$.
 	\end{enumerate}
 \end{lemma}		
 \begin{proof}
 	We shall use  Lemma \ref{lm expressions of units}  and keep its notations. Note that   $\{\varepsilon_{2},   \varepsilon_{pq},
 	\sqrt{\varepsilon_{pq}\varepsilon_{2pq}}\}$ is a FSU of $L_{pq}$. 
 	\begin{enumerate}[$\bullet$]
 		\item     	$N_{L_{pq}^*/L_{pq}}(\sqrt{\varepsilon_{q}})=\frac{1}{\sqrt{2}}(d_1'+d_2'\sqrt{q})\cdot \frac{1}{\sqrt{2}}(d_1'-d_2'\sqrt{q})=1$.\\
 		$N_{L_{pq}^*/L_{pq}}(\sqrt{\varepsilon_{pq}})=\frac{1}{\sqrt{2}}(b_1\sqrt{2p}+b_2\sqrt{q})\cdot \frac{1}{\sqrt{2}}(-b_1\sqrt{2p}-b_2\sqrt{q})
 		=-\sqrt{\varepsilon_{pq}}\cdot\sqrt{\varepsilon_{pq}}=- \varepsilon_{pq}$.
 		\item The norms in the second point are direct.
 		\item  We have $N_{L_{pq}^*/L_{pq}}( {\varepsilon_{2}\varepsilon_{p}\varepsilon_{2p}})=\varepsilon_{2}^2\cdot (-1)\cdot (-1)$. Thus,
 		$N_{L_{pq}^*/L_{pq}}(\sqrt{\varepsilon_{2}\varepsilon_{p}\varepsilon_{2p}})=\pm \varepsilon_{2}$. \\
 		Since $N_{L_{pq}^*/L_{pq}}(\sqrt{\varepsilon_{2q}})=\frac{1}{\sqrt{2}}(d_1+d_2\sqrt{2q})\cdot \frac{1}{\sqrt{2}}(d_1-d_2\sqrt{2q})=1$,  then
 		$N_{L_{pq}^*/L_{pq}}(\zeta_8\sqrt{ 	\varepsilon_{q}}\sqrt{\varepsilon_{2q}})=\zeta_8^2\cdot 1\cdot 1$. Thus $N_{L_{pq}^*/L_{pq}}(\sqrt[4]{\zeta_8^2	\varepsilon_{q}\varepsilon_{2q}})=\pm\zeta_8 $.\\
 		$N_{L_{pq}^*/L_{pq}}(\varepsilon_{2}\sqrt{{\varepsilon_{q}}}\sqrt{{\varepsilon_{pq}}\varepsilon_{2pq}})=  \varepsilon_{2}^2\cdot 1 \cdot  \varepsilon_{pq}\cdot\varepsilon_{2pq}$.
 		Thus,  $N_{L_{pq}^*/L_{pq}}(\sqrt[4]{\varepsilon_{2}^2{\varepsilon_{q}}{\varepsilon_{pq}\varepsilon_{2pq}}})=\pm \varepsilon_{2}\sqrt{\varepsilon_{pq}\varepsilon_{2pq}}$.
 	\end{enumerate}
 \end{proof}
 \begin{theorem}\label{thm G_L quaternion}Let $p$ and $q$ be two primes satisfying    conditions {\eqref{cond 1}} or $ {\eqref{cond 2}} $ and let $p'$ be a prime satisfying  conditions  $\textcolor{blue}{\eqref{cond 3}}$. Let $m\geq 2$ be an integer such that $h_2(-2q)=2^m$. Then
 	the group $G_{L_{pq}}\simeq Q_{m+1}$ and  the group $G_{L_{p'}}$ is  of type  $(2,  2)$.

 \end{theorem}	
 \begin{proof}
 	\begin{enumerate}[\rm 1.]
 		\item   By Proposition \ref{prop the order of G_k} and Theorem \ref{class group of L*} we have $|G_{L_{pq}}|=2\cdot h_2(L_{pq}^*)=2^{m+1}$. Assume that  $p$ and $q$  verify  conditions {\eqref{cond 1}}.
 		According to \cite[Corollaire 17]{Be05},   the group $G_{k_{pq}}$ is quaternion or semidihedral.
 		By Galois theory we have $$Gal(L_{pq}^{(2)}/k_{pq})\simeq Gal(L_{pq}/k_{pq})\times G_{L_{pq}}.$$
 		Thus $G_{L_{pq}}$ is a subgroup of $G_{k_{pq}}$ of index $2$. Therefore  $G_{L_{pq}}$  is dihedral or quaternion. Since a FSU of $L_{pq}$ is $\{\varepsilon_{2},   \varepsilon_{pq},
 		\sqrt{\varepsilon_{pq}\varepsilon_{2pq}}\}$ (cf.\cite{chemsZkhnin1}),   then by \cite{Heider},  Theorem \ref{thm main theorem on units} and Lemma \ref{norms  case T},   the number of classes of $\mathrm{Cl}_2(L_{pq})$ which capitulate in $L_{pq}^*$ is	$[L_{pq}^*:L_{pq}][E_{L_{pq}}:N_{L_{pq}^*/L_{pq}}(E_{L_{pq}^*})]=2\cdot 1=2$.
 		So from Table \ref{tablecapi},  we deduce that $G_{L_{pq}}$ can not be dihedral. Hence $G_{L_{pq}}$ is quaternion.
 		
 		Suppose now that $p$ and $q$ verify the condition $ {\eqref{cond 2}} $. As previously we show that $G_{L_{pq}}$ is a subgroup of $G_{k_{pq}}$ of index $2$ and by \cite{Be05} $G_{k_{pq}}$ is quaternion. So  $G_{L_{pq}}$ is quaternion.
 		\item Let   $k'=\mathbb{Q}(\sqrt{-1},   \sqrt{2p'}) $. As $L_{p'}=\mathbb{Q}(\sqrt{-1},  \sqrt{ p'},   \sqrt{ 2})$ is the genus field of $k'$, so 
 		by \cite[Théorème 5.2]{azizitaous(24)(222)}, the $2$-class group of $k'$ is of type $(2,  4)$. Hence
 		\cite[Corollaire 1]{taous2008} implies that 
 		the Hilbet $2$-class field tower of $k$ terminates at the first step. Therefore,   the Hilbert $2$-class field tower of its genus field $L_{p'}$ terminates at the first step. So the result.
 	\end{enumerate}	
 \end{proof}	
 \begin{corollary} Keep  the assumptions of the previous Theorem $\ref{thm G_L quaternion}$. Then
 	\begin{enumerate}[\rm 1.]
 		\item There are exactly $2$ classes of $\mathrm{Cl}_2(L_{pq})$ which capitulate in each of the three unramified quadratic extensions of $L_{pq}$.
 		\item There are $4$ classes of $\mathrm{Cl}_2(L_{p'})$ which capitulate in each of the three unramified quadratic extensions of $L_{p'}$.
 	\end{enumerate}
 \end{corollary}	
 
 \begin{lemma}Let $p$ and $q$ be two primes satisfying    conditions {\eqref{cond 1}}. Then
 	\begin{enumerate}[$\bullet$]\label{norms2  case T}
 		\item    $N_{L_{pq}^*/F_{pq}}(\sqrt{\varepsilon_{q}})= -1$ and    $N_{L_{pq}^*/F_{pq}}(\sqrt{\varepsilon_{pq}})=\varepsilon_{pq}$.
 		\item   $N_{L_{pq}^*/F_{pq}}(\varepsilon_{2})=-1$,   $N_{L_{pq}^*/F_{pq}}(\varepsilon_{p})=-1$ and $N_{L_{pq}^*/F_{pq}}(\zeta_8)=-1$.
 		\item $N_{L_{pq}^*/F_{pq}}(\sqrt{\varepsilon_{2}\varepsilon_{p}\varepsilon_{2p}})=\pm \varepsilon_{2p}$,
 		$N_{L_{pq}^*/F_{pq}}(\sqrt[4]{\zeta_8^2	\varepsilon_{q}\varepsilon_{2q}})=\pm\sqrt{-\varepsilon_{2q}} $ and   $N_{L_{pq}^*/F_{pq}}(\sqrt[4]{\varepsilon_{2}^2{\varepsilon_{q}}{\varepsilon_{pq}\varepsilon_{2pq}}})=\pm    \sqrt{ \varepsilon_{pq}}$.
 	\end{enumerate}
 \end{lemma}	
 \begin{proof}
 	\begin{enumerate}[$\bullet$]
 		\item 	Let us  use and keep the notations notations of Lemma \ref{lm expressions of units}. We have:\\  $N_{L_{pq}^*/F_{pq}}(\sqrt{\varepsilon_{q}})=\frac{1}{\sqrt{2}}(d_1'+d_2'\sqrt{q})\cdot \frac{1}{-\sqrt{2}}(d_1'-d_2'\sqrt{q})=\frac{-1}{2}(d_1'^2-d_2'^2{q})=-1$.\\
 		$N_{L_{pq}^*/F_{pq}}(\sqrt{\varepsilon_{pq}})=\frac{1}{\sqrt{2}}(b_1\sqrt{2p}+b_2\sqrt{q})\cdot \frac{1}{-\sqrt{2}}(-b_1\sqrt{2p}-b_2\sqrt{q})
 		=\sqrt{\varepsilon_{pq}}\cdot\sqrt{\varepsilon_{pq}}= \varepsilon_{pq}$.
 		\item 	The first norm and the second norm in the second point are direct. We have
 		$N_{L_{pq}^*/F_{pq}}(\zeta_8)=N_{L_{pq}^*/F_{pq}}(\frac{1}{\sqrt{2}}(1+i))=\frac{-1}{ {2}}(1-i^2)=-1$.
 		
 		\item 	Note that $N_{L_{pq}^*/L_{pq}}( { \varepsilon_{2p}})=\varepsilon_{2p}^2$. Then $N_{L_{pq}^*/L_{pq}}( {\varepsilon_{2}\varepsilon_{p}\varepsilon_{2p}})=(-1)\cdot  (-1)\cdot \varepsilon_{2p}^2$. Thus	$N_{L_{pq}^*/L_{pq}}(\sqrt{\varepsilon_{2}\varepsilon_{p}\varepsilon_{2p}})=\pm \varepsilon_{2p}$.\\
 		We have $N_{L_{pq}^*/F_{pq}}(\sqrt{\varepsilon_{2q}})=\frac{1}{\sqrt{2}}(d_1+d_2\sqrt{2q})\cdot \frac{1}{-\sqrt{2}}(d_1+d_2\sqrt{2q}) =-\varepsilon_{2q}$,  then
 		$N_{L_{pq}^*/F_{pq}}(\zeta_8\sqrt{ 	\varepsilon_{q}}\sqrt{\varepsilon_{2q}})=(-1)\cdot (-1)\cdot (-\varepsilon_{2q})=-\varepsilon_{2q}$. Therefore $N_{L_{pq}^*/F_{pq}}(\sqrt[4]{\zeta_8^2	\varepsilon_{q}\varepsilon_{2q}})= \pm \sqrt{-\varepsilon_{2q}}$.\\
 		We have $N_{L_{pq}^*/F_{pq}}(\sqrt{\varepsilon_{2pq}})=\frac{1}{\sqrt{2}}(y_1\sqrt{2p}+y_2\sqrt{q})\cdot \frac{1}{-\sqrt{2}}(y_1\sqrt{2p}-y_2\sqrt{q})
 		=\frac{-1}{2} (y_1^2{2p}-y_2^2{q})=1$. So
 		
 		$N_{L_{pq}^*/F_{pq}}(\varepsilon_{2}\sqrt{{\varepsilon_{q}}}\sqrt{{\varepsilon_{pq}}}\sqrt{\varepsilon_{2pq}})=  (-1)\cdot (-1) \cdot    \varepsilon_{pq}  \cdot 1=\varepsilon_{pq}$.
 		Therefore  $N_{L_{pq}^*/F_{pq}}(\sqrt[4]{\varepsilon_{2}^2{\varepsilon_{q}}{\varepsilon_{pq}\varepsilon_{2pq}}})=\pm  \sqrt{ \varepsilon_{pq}}$.
 	\end{enumerate}	
 \end{proof}
 
 \begin{theorem}\label{th 2 on G_F}
 	Let $m\geq 2$ such that $h_2(-2q)=2^m$.
 	\begin{enumerate}[\rm 1.]
 		\item Let $p$ and $q$ be two primes satisfying    conditions {\eqref{cond 1}}. Then  $G_{F_{pq}}\simeq D_{m+1}$.
 		\item Let $p$ and $q$ be two primes satisfying   conditions $ {\eqref{cond 2}} $. Then $G_{F_{pq}}\simeq Q_{m+1}$.
 	\end{enumerate}
 \end{theorem}
 \begin{proof}
 	\begin{enumerate}[\rm 1.]
 		\item Since, by the third point of Lemma \ref{norms2  case T},  $\sqrt{\varepsilon_{ pq}}\in F_{pq}$,  then according to \cite[Proposition 5]{Be05},  a FSU of $F_{pq}$ is given by
 		$\{\sqrt{\varepsilon_{ pq}}, \varepsilon_{2p}, \sqrt{-\varepsilon_{2q}}\}$. As in the proof of  Theorem \ref{thm G_L quaternion} and using the same references we deduce that  $G_{F_{pq}}\simeq Q_{m+1}$ or $D_{m+1}$.
 		By Lemma \ref{norms2  case T},    \cite{Heider} and  Theorem \ref{thm main theorem on units} ,   the number of classes of $\mathrm{Cl}_2(F_{pq})$ which capitulate in $L_{pq}^*$ is	$[L_{pq}^*:F_{pq}][E_{F_{pq}}:N_{L_{pq}^*/F_{pq}}(E_{L_{pq}^*})]=2\cdot 2=4$. So the first item.
 		\item The proof of the second item is similar to that  of Theorem \ref{thm G_L quaternion}.
 	\end{enumerate}
 \end{proof}
 \begin{corollary}~\
 	\begin{enumerate}[\rm 1.]
 		\item Let $p$ and $q$ be two primes satisfying    conditions {\eqref{cond 1}}. Then there are  $4$ classes of $\mathrm{Cl}_2(F_{pq})$ which capitulate in  $L_{pq}^*$ and there are exactly $2$ classes of $\mathrm{Cl}_2(F_{pq})$ which capitulate in each of the other quadratic unramified extensions of $F_{pq}$.
 		\item Let $p$ and $q$ be two primes satisfying   conditions ${\eqref{cond 2}}$.  Then,   there are exactly $2$ classes of $\mathrm{Cl}_2(F_{pq})$ which capitulate in each of the three unramified quadratic extensions of $F_{pq}$.
 	\end{enumerate}
 	
 \end{corollary}	
 
 \begin{remark}
 	\begin{enumerate}[\rm $\bullet$]

 		\item  Assume  that $p$ and $q$ verify conditions {\eqref{cond 1}} or   {\eqref{cond 2}}. The authors of \cite{Be05} didn't determine the order  of $G_{k_{pq}}$,  but now by the above results it is easy to see that  $|G_{k_{pq}}|=2^{m+2}$,   with   $m\geq 2$ such that $h_2(-2q)=2^m$, and so we have the following diagram (see Figure \ref{Fig 2}):

 		\begin{center}
 			\begin{figure}[H]\label{ diagram Lp/Qi}
 				\hspace*{5cm}
 				\begin{minipage}{5cm}
 					{\footnotesize
 						\hspace{0.5cm}\begin{tikzpicture} [scale=1.2]
 						\node (k)  at (0,  0) {$k_{pq}$};
 						\node (K)  at (-1,  1) {$K_{pq}$};
 						\node (F)  at (1,  1) {$F_{pq}$};
 						\node (L)  at (0,  1) {$L_{pq}$};
 						\node (L*)  at (0,  2) {$L_{pq}^*=k_{pq}^{(1)}$};
 						\node (Lpq1)  at (0,  3) {$L_{pq}^{(1)}=F_{pq}^{(1)}$};
 						\node (Lpq1*1)  at (0,  5.5) {${{L_{pq}^*}}^{(1)}=F_{pq}^{(2)}=L_{pq}^{(2)}=K_{pq}^{(1)}=K_{pq}^{(2)}=k_{pq}^{(2)}$};
 						\draw (k) --(K)  node[scale=0.4,  midway,  below right]{};
 						\draw (k) --(F)  node[scale=0.4,  midway,  below right]{};
 						\draw (k) --(L)  node[scale=0.4,  midway,  below right]{\Large \bf2};
 						\draw (k) --(L)  node[scale=0.4,  midway,  below right]{};
 						\draw (L) --(L*)  node[scale=0.4,  midway,  below right]{\Large \bf2};
 						\draw (K) --(L*)  node[scale=0.4,  midway,  below right]{};
 						\draw (F) --(L*)  node[scale=0.4,  midway,  below right]{};
 						\draw (L*) --(Lpq1)  node[scale=0.4,  midway,  below right]{\Large \bf \hspace{-0.2cm}\;\;\;2};
 						\draw (Lpq1) --(Lpq1*1)  node[scale=0.4,  midway,  below right]{\Large \bf2$^{\text{m-1 }}$};
 						\end{tikzpicture}}
 				\end{minipage}
 				\caption{The Hilbert $2$-class field towers.}\label{Fig 2}
 			\end{figure}
 		\end{center}
 		\item Note also that the authors of \cite{Be05}   didn't determine the exact structure  of $G_{k_{pq}}$ whenever $p$ and $q$ satisfy conditions {\eqref{cond 1}}. Now by our above  results it is easy to see that, under conditions  {\eqref{cond 1}}, $G_{k_{pq}}$ is semidihedral of order $2^{m+2}$.
 	\end{enumerate}
 \end{remark}

\end{document}